\documentclass[12pt]{amsart}
\usepackage[all]{xy}
\usepackage[parfill]{parskip}

\usepackage{color}

\usepackage{amsmath, amscd, graphicx, latexsym, hyperref, times}
\usepackage{color,graphicx}
\usepackage[abs]{overpic}
\usepackage{tikz}
\usepackage{tikz-cd}
\usetikzlibrary{arrows, patterns}

\oddsidemargin .25in \theoremstyle{plain}
\newtheorem{dummy}{anything}[section]
\newtheorem{theorem}[dummy]{Theorem}

\newtheorem{lemma}[dummy]{Lemma}

\newtheorem{proposition}[dummy]{Proposition}

\newtheorem{corollary}[dummy]{Corollary}

\theoremstyle{definition}
\newtheorem{definition}[dummy]{Definition}

\newcommand{\Z}{\mathbb{Z}}

\newcommand{\HF}{\widehat{HF}}
\newcommand{\HFK}{\widehat{HFK}}
\newcommand{\CFK}{\widehat{CFK}}

\begin{document}

\title{A bound for rational Thurston-Bennequin invariants}

\author{Youlin Li and Zhongtao Wu}

\address{School of Mathematical Sciences, Shanghai Jiao Tong University, Shanghai 200240, China}

\email{liyoulin@sjtu.edu.cn}
\address{Department of Mathematics, The Chinese University of Hong Kong, Shatin, Hong Kong}
\email{ztwu@math.cuhk.edu.hk}

\subjclass[2000]{}



\begin{abstract}
In this paper, we introduce a rational $\tau$ invariant for rationally null-homologous knots in  contact 3-manifolds with nontrivial Ozsv\'{a}th-Szab\'{o} contact invariants. Such an invariant is an upper bound for the sum of rational Thurston-Bennequin invariant and the rational rotation number of the Legendrian representatives of the knot.  In the special case of Floer simple knots in L-spaces, we can compute the  rational $\tau$ invariants by correction terms.
\end{abstract}

\maketitle

\section{Introduction}\label{sec: intro}

Given a Legendrian  representative $L$ of an integrally null-homologous knot $K$ in a tight contact  3-manifold $(Y, \xi)$.  We have the well-known Bennequin-Eliashberg inequality \cite{b} \cite{e} $$tb(L)+rot(L)\leq 2g(K)-1,$$ where $g(K)$ is the genus of $K$.   Plameneveskaya \cite{p} improved this inequality for knots in the tight contact 3-sphere $(S^{3}, \xi_{std})$, and showed that  $$tb(L)+rot(L)\leq 2\tau(K)-1,$$  where $\tau(K)$ is an invariant of $K$ defined by Ozsv\'{a}th and Szab\'{o} \cite{OSzFourBallGenus}.  Later on, Hedden \cite{h} introduced an invariant $\tau_{\xi}(K, F)$ for an integrally null-homologous knot $K$ with a Seifert surface $F$ in a contact 3-manifold $(Y, \xi)$ with a non-trivial Ozsv\'{a}th-Szab\'{o} contact invariant  $c(\xi)$ \cite{OSzContact}.  He proved that for any Legendrian representatives $L$ of $K$ in $(Y,\xi)$, $$tb(L)+rot(L; F)\leq 2\tau_{\xi}(K, F)-1.$$

More generally, consider a rationally null-homologous knot $K$ in a 3-manifold $Y$. Let $L$ be a Legendrian representative of a rationally null-homologous knot $K$ in a contact 3-manifold $(Y,\xi)$, and let $F$ be a rational Seifert surface of $K$.  Baker and Etnyre \cite{be} defined the rational Thurston-Bennequin invariant $tb_{\mathbb{Q}}(L)$ and rational rotation number $rot_{\mathbb{Q}}(L; F)$.
When $\xi$ is a tight contact structure on $Y$, they showed that
\begin{equation}\label{BEineq}
tb_{\mathbb{Q}}(L)+rot_{\mathbb{Q}}(L; F)\leq -\frac{1}{q}\chi(F),
\end{equation}
where $q$ is the order of $[K]$ in $H_{1}(Y;\mathbb{Z})$.


In this paper, we introduce an invariant $\tau^{\ast}_{c(\xi)}(Y, K, F)$ for an rationally null-homologous knot $K$, which generalizes Hedden's definition \cite{h}. 
Our main theorem proves that this invariant gives an upper bound for the sum of the rational Thurston-Bennequin invariant and the rational rotation number of all Legendrian representatives of $K$.

\begin{theorem} \label{thm: main} Suppose $K$ is a rationally null-homologous knot in a 3-manifold $Y$ with a rational Seifert surface $F$, and $\xi$ is a contact structure on $Y$ with nontrivial Ozsv\'{a}th-Szab\'{o} contact invariant $c(\xi)\in \widehat{HF}(-Y, \mathfrak{s}_{\xi})$. Then for any Legendrian representative $L$ of $K$, we have \begin{equation}\label{plus-1}
tb_{\mathbb{Q}}(L)+rot_{\mathbb{Q}}(L; F)\leq 2\tau^{\ast}_{c(\xi)}(Y, K, F)-1.
\end{equation}
\end{theorem}

\bigskip
A closed 3-manifold $Y$ is called an \emph{L-space} if it is a rational homology sphere and $\text{rank}\widehat{HF}(Y)$ $=$ $|H_{1}(Y)|$. A knot $K$ in an L-space $Y$ is called \emph{Floer simple} if $\text{rank}\widehat{HFK}(Y,K)$ $=$ $\text{rank}\widehat{HF}(Y)$. Our next result shows that the rational $\tau$ invariant of a Floer simple knot in an L-space $Y$ can be expressed in terms of the correction terms of $Y$; in particular, it depends only on the order of the knot (rather than its isotopy class).

\begin{proposition}\label{prop:tau}

For a Floer simple knot $K$ in an L-space $Y$,

$$2\tau_\mathfrak{s}(Y,K) = d(Y, \mathfrak{s})-d(Y, J\mathfrak{s}+\text{PD}[K]).$$

\end{proposition}

While the precise definition of $\tau_\mathfrak{s}(Y,K)$ will be given later, we remark that $\tau_{\mathfrak{s}_{\xi}}(Y, K)=\tau^{\ast}_{c(\xi)}(Y, K, F)$ when $Y$ is an L-space with a nontrivial Ozsv\'{a}th-Szab\'{o} contact invariant $c(\xi)$ in the Spin$^c$ structure $\mathfrak{s}_{\xi}$.
Also note that $rot_{\mathbb{Q}}(L; F)$ is independent of $F$ when $Y$ is a rational homology sphere, and it may be abbreviated as $rot_{\mathbb{Q}}(L)$.
We have the following immediate corollary.

\begin{corollary}\label{cor1}
Suppose  $K$ is a Floer simple knot in an L-space $Y$, $\xi$ is a contact structure on $Y$ with nontrivial Ozsv\'{a}th-Szab\'{o} contact invariant $c(\xi)\in \widehat{HF}(-Y, \mathfrak{s}_{\xi})$. Then for any Legendrian representative $L$ of $K$, $$tb_{\mathbb{Q}}(L)+rot_{\mathbb{Q}}(L)\leq d(Y, \mathfrak{s}_{\xi})-d(Y, J(\mathfrak{s}_{\xi}+PD[K]))-1.$$
\end{corollary}

\bigskip
The remaining part of this paper is organized as follows.  In Section 2, we review Alexander filtration on knot Floer complex and use it to define a rational $\tau$ invariant associated to a knot in a 3-manifold possessing non-vanishing Floer (co)homology classes.  In Section 3, we recall the notions of rational Thurston-Bennequin invariant and rational rotation number. In particular, we exhibit how these two invariants behave under connected sum of two Legendrian knots.  In Section 4, we prove Theorem \ref{thm: main}.  In Section 5, we study in more detail the case of Floer simple knots in L-spaces. We show that rational $\tau$ invariants are determined by the correction terms. In Section 6, we specialize further to an example of Legendrian representatives of simple knots in lens spaces.

\vspace{5pt}\noindent{\bf Acknowledgements.}  This work was carried out while the first author was visiting the Chinese University of Hong Kong and he would like to thank for their hospitality. The first author was partially supported by grant no.\ 11471212 of the National Natural Science Foundation of China.  The second author was partially supported by grant from the Research Grants Council of the Hong Kong Special Administrative Region, China (Project No.\ 24300714).

\section{Rational $\tau$ invariants} \label{sec: invariant}

Let $K$ be a knot in $Y$ and $(\Sigma, \boldsymbol{\alpha}, \boldsymbol{\beta}, w, z)$ be a corresponding doubly pointed Heegaard diagram. Then the set of relative Spin$^{c}$-structures determine a filtration of the chain complex $\widehat{CF}(Y)$ via a map
$$\mathfrak{s}_{w,z}:\mathbb{T}_{\alpha}\cap\mathbb{T}_{\beta}\rightarrow \underline{\text{Spin}^{c}}(Y,K).$$
Each relative Spin$^{c}$ structure $\underline{\mathfrak{s}}$ for $(Y,K)$ corresponds to a Spin$^{c}$ structure $\mathfrak{s}$ on $Y$ via a natural map $G_{Y,K}:\underline{\text{Spin}^{c}(}Y,K)\rightarrow \text{Spin}^{c}(Y)$.  

From now on, assume that $K$ is a rationally null-homologous  knot in a 3-manifold $Y$, and $[K]$ is of order $q$ in $H_{1}(Y;\mathbb{Z})$. A \emph{rational Seifert surface}  for $K$ is defined to be a map $j:F\rightarrow Y$ from a connected compact orientable surface $F$ to $Y$ that is an embedding of the interior of $F$ into $Y\setminus K$, and a $q$-fold cover  from its boundary $\partial F$ to $K$.  Let $N(K)$ be a tubular neighborhood of $K$ in $Y$, and $\mu\subset \partial N(K)$ the meridian of $K$.  We can assume that $F\cap \partial N(K)$ consists of $c$ parallel cooriented simple closed curves, each of which has homology $[\nu]\in H_{1}(\partial N(K); \mathbb{Z})$. We can choose a canonical longitude $\lambda_{can}$ such that $[\nu]=t[\lambda_{can}]+r[\mu]$, where $t$ and $r$ are coprime integers, and $0\leq r<t$. Note that $ct=q$.

Suppose $K$ corresponds to a doubly pointed Heegaard diagram $(\Sigma, \boldsymbol{\alpha}, \boldsymbol{\beta}, w, z)$.  Fix a rational Seifert surface $F$ for $K$.  Following Ni \cite
{n},\footnote{Ni's original definition assumes that $Y$ is a rational homology sphere.} we define the Alexander grading of a relative Spin$^{c}$-structure $\underline{\mathfrak{s}} \in \underline{\text{Spin}^{c}(}Y,K) $ by

\begin{equation}\label{AlexGrading}
A_{F}(\underline{\mathfrak{s}})=\frac{1}{2q}(\langle  c_{1}(\underline{\mathfrak{s}}), [\tilde{F}] \rangle-q),
\end{equation}
where $\tilde{F}$ is the closure of $j(F)\setminus N(K)$.

Moreover, the Alexander grading of an intersection point $x\in \mathbb{T}_{\alpha}\cap\mathbb{T}_{\beta}$  is defined by
$$A_{F}(x)=\frac{1}{2q}(\langle c_{1}(\mathfrak{s}_{w,z}(x)), [\tilde{F}]\rangle-q).$$

In general, the Alexander grading $A_{F}$ takes values in rational number $\mathbb{Q}$.  Nonetheless, observe that for any two relative Spin$^c$ structures $\underline{\mathfrak{s}}_1, \underline{\mathfrak{s}}_2 \in G_{Y,K}^{-1} (\mathfrak{s})$ of a fixed $\mathfrak{s}$, we have $\underline{\mathfrak{s}}_2-\underline{\mathfrak{s}}_1=k \, \text{PD}[\mu]$ for some integer $k$.  Hence, there exists a unique rational number $k_{\mathfrak{s}, F}\in[-\frac{1}{2}, \frac{1}{2})$ depending only on $\mathfrak{s}$ and $F$ such that for every $\underline{\mathfrak{s}}\in G_{Y,K}^{-1}(\mathfrak{s})$,  $$\frac{1}{2q}(\langle c_{1}(\underline {\mathfrak{s}}), [\tilde{F}]\rangle-q)=k_{\mathfrak{s}, F}+k.$$
for some integer $k$ \cite{r}.

As a result, the Alexander grading induces effectively a $\mathbb{Z}$-filtration of $\widehat{CF}(Y, \mathfrak{s})$ by
$$\mathcal{F}_{\mathfrak{s},k}=\{x\in\widehat{CF}(Y, \mathfrak{s})|A_{F}(x)\leq k_{\mathfrak{s}, F}+k\},$$ where $k\in\mathbb{Z}$.
Let  $i_{k}:\mathcal{F}_{\mathfrak{s},k}\rightarrow \widehat{CF}(Y, \mathfrak{s})$  be the inclusion map. It induces a homomorphism between the homologies $I_{k}:H_{\ast}(\mathcal{F}_{\mathfrak{s},k})\rightarrow \widehat{HF}(Y, \mathfrak{s}).$

Next we introduce two rational $\tau$ invariants in the same way as Hedden did for integrally null-homologous knots  \cite{h}.

\begin{definition}\label{def:tau}
For any $[x]\neq0\in \widehat{HF}(Y, \mathfrak{s})$, define
$$\tau_{[x]}(Y,K, F)=\text{min}\{k_{\mathfrak{s}, F}+k|[x]\in \text{Im}(I_{k})\}.$$
\end{definition}

Consider the orientation reversal $-Y$ of $Y$,  we have the paring
$$\langle-,-\rangle:\widehat{CF}(-Y,\mathfrak{s})\otimes\widehat{CF}(Y,\mathfrak{s})\rightarrow \mathbb{Z}/2\mathbb{Z},$$
given by
$$\langle x,y\rangle=
\begin{cases}
1& \text{if}~~   x=y,\\
0& \text{otherwise}.
\end{cases}$$
It descends to a pairing
$$\langle-,-\rangle:\widehat{HF}(-Y,\mathfrak{s})\otimes\widehat{HF}(Y,\mathfrak{s})\rightarrow \mathbb{Z}/2\mathbb{Z}.$$

\begin{definition}\label{def:tau*}
For any $[y]\neq0\in \widehat{HF}(-Y, \mathfrak{s})$, define
$$\tau^{\ast}_{[y]}(Y, K, F)=\text{min}\{k_{\mathfrak{s}, F}+k|\exists\alpha\in \text{Im}(I_{k}), \text{such that} \langle[y], \alpha\rangle\neq 0\}.$$
\end{definition}

Using the same argument as in the proof of \cite[Proposition 28]{h}, we have the following duality.

\begin{proposition} \label{prop: dual}
Let $[y]\neq0\in\widehat{HF}(-Y,\mathfrak{s})$. Then
$$\tau_{[y]}(-Y, K, F)=-\tau^{\ast}_{[y]}(Y, K, F).$$
\end{proposition}

For $i=1, 2$, let $K_{i}$ be a rationally null-homologous  knot in a 3-manifold $Y_{i}$ with order $q_{i}$, and $j:F_{i}\rightarrow Y_{i}$ be a rational Seifert surface for $K_{i}$.   Let $K_{1}\sharp K_{2}$ denote their connected sum in $Y_{1}\sharp Y_{2}$. Then the order of $K_{1}\sharp K_{2}$ is $\mathrm{lcm}(q_1, q_2)$, that is, the least common multiple of $q_1$ and $q_2$.   One can construct a rational Seifert surface for $K_{1}\sharp K_{2}$ by taking $\frac{\mathrm{lcm}(q_1, q_2)}{q_{1}}$ copies of $j:F_{1}\rightarrow Y_{1}$ and $\frac{\mathrm{lcm}(q_1, q_2)}{q_{2}}$ copies of $j:F_{2}\rightarrow Y_{2}$ and gluing them in an appropriate way. See the next section. We denote it by $j:F_{1}\natural F_{2}\rightarrow Y_{1}\sharp Y_{2}$.

By \cite[Lemma 3.8]{r}, for $x_{1}\in \widehat{CF}(Y_{1})$ and $x_{2}\in \widehat{CF}(Y_{2})$, we have $$A_{F_{1}\natural F_{2}}(x_{1}\otimes x_{2})=A_{F_{1}}(x_{1})+A_{F_{2}}(x_{2}).$$ So we can use the same argument as in the proof of  \cite[Proposition 29]{h} to obtain the following proposition.

\begin{proposition} \label{prop: additivity}
 For any $[x_{i}]\neq0\in\widehat{HF}(Y_{i}, \mathfrak{s}_{i})$, $[y_{i}]\neq0\in\widehat{HF}(-Y_{i}, \mathfrak{s}_{i})$, $i=1,2$, we have
$$\tau_{[x_1]\otimes[x_2]}(Y_{1}\sharp Y_{2}, K_{1}\sharp K_{2}, F_{1}\natural F_{2})=\tau_{[x_1]}(Y_{1}, K_{1}, F_{1})+\tau_{[x_2]}(Y_{2}, K_{2}, F_{2}),$$
and
$$\tau^{\ast}_{[y_1]\otimes[y_2]}(Y_{1}\sharp Y_{2}, K_{1}\sharp K_{2}, F_{1}\natural F_{2})=\tau^{\ast}_{[y_1]}(Y_{1}, K_{1}, F_{1})+\tau^{\ast}_{[y_2]}(Y_{2}, K_{2}, F_{2}).$$
\end{proposition}

Let  $X_{-n}(K)$ be the cobordism from $Y$ to $Y_{-n}(K)$ obtained by attaching a 4-dimensional 2-handle to $K\times {1}\subset Y\times [0,1]$ with $(-n)$-framing with respect to the canonical longitude.  Suppose $\mathfrak{r}_{k}$ is the restriction to $Y_{-n}(K)$ of the unique Spin$^{c}$ structure $\mathfrak{t}_{k}$ on $X_{-n}(K)$ satisfying $\mathfrak{t}_{k}|_{Y}=\mathfrak{s}$ and
$$\langle c_{1}(\mathfrak{t}_{k}),[\tilde{F}\cup qC]\rangle-nq-cr=2q(k_{\mathfrak{s}, F}+k),$$ where $C$ is the core of the added 2-handle in $X_{-n}(K)$, and $[\tilde{F}\cup qC]$ is a generator of $H_{2}(X_{-n}(K);\mathbb{Z})\cong\mathbb{Z}$.
We have the following homomorphism between homology induced by the above cobordism
$$\hat{F}^{\mathfrak{s}}_{-n,k} :\widehat{HF}(Y,\mathfrak{s})\rightarrow \widehat{HF}(Y_{-n}(K), \mathfrak{r}_{k}).$$

By \cite[Theorem 4.2]{r}, we have a commutative diagram
$$\xymatrix{
    \widehat{CF}(Y,\mathfrak{s}) \ar[rr]^(0.45){f^{\mathfrak{s}}_{-n,k}}\ar[d]_{\cong} & & \widehat{CF}(Y_{-n}(K), \mathfrak{r}_{k})\ar[d]^{\cong} \\
     \mathcal{C}_{\mathfrak{s}}\{i=0\}\ar[rr]^(0.4){f^{\mathfrak{s}}_{-n,k}}&  & \mathcal{C}_{\mathfrak{s}}\{\text{min}(i,j-k)=0\}
  }$$ where $f^{\mathfrak{s}}_{-n,k}$ induces the map $\hat{F}^{\mathfrak{s}}_{-n,k}$ on homologies.
We then apply the argument of \cite[Proposition 24]{h} and \cite[Proposition 26]{h} to obtain the following two  propositions.

\begin{proposition} \label{prop: cobordismmap1}
Let $[x]\neq0\in \widehat{HF}(Y,\mathfrak{s})$ and $n>0$ be sufficiently large. We have\\
(1) If $k_{\mathfrak{s}, F}+k<\tau_{[x]}(Y, K, F)$, then $\hat{F}^{\mathfrak{s}}_{-n,k}([x])\neq0.$\\
(2) If $k_{\mathfrak{s}, F}+k>\tau_{[x]}(Y, K, F)$, then $\hat{F}^{\mathfrak{s}}_{-n,k}([x])=0.$
\end{proposition}

\begin{proposition} \label{prop: cobordismmap2}
Let $[y]\neq0\in \widehat{HF}(-Y,\mathfrak{s})$ and $n>0$ be sufficiently large. We have\\
(1) If $k_{\mathfrak{s}, F}+k<\tau^{\ast}_{[y]}(Y, K, F)$, then for every $\alpha\in\widehat{HF}(Y,\mathfrak{s})$ such that $\langle[y],\alpha\rangle\neq0$, we have $\hat{F}^{\mathfrak{s}}_{-n,k}(\alpha)\neq0.$\\
(2) If $k_{\mathfrak{s}, F}+k>\tau^{\ast}_{[y]}(Y, K, F)$, then there exists $\alpha\in\widehat{HF}(Y,\mathfrak{s})$ such that $\langle[y],\alpha\rangle\neq0$ and $\hat{F}^{\mathfrak{s}}_{-n,k}(\alpha)=0.$
\end{proposition}

\section{Rationally null-homologous Legendrian knots} \label{sec: legendrian}

Given a rationally null-homologous oriented Legendrian knot $L$ in a contact 3-manifold $(Y, \xi)$.  Suppose that its order is $q$, and it has a rational Seifert surface $j:F\rightarrow Y$. The \emph{rational Thurston-Bennequin invariant} of $L$,  $tb_{\mathbb{Q}}(L)$, is defined to be $\frac{1}{q} L'\cdot j(F)$, where $L'$ is a copy of $L$ obtained by pushing off using the framing coming from $\xi$, and $\cdot$ denotes the algebraic intersection number.  We fix a trivialization $F\times \mathbb{R}^{2}$  of the pullback bundle $j^{\ast}\xi$ on $F$. The restriction of $\xi$ on L is $\xi|_{L}=L\times \mathbb{R}^{2}$ and has a section $TL$. The pullback $j^{\ast}(TL)$ is a section of  $\partial F\times \mathbb{R}^{2}$. The \emph{rational rotation number} of $L$,  $rot_{\mathbb{Q}}(L)$, is defined to be the winding number of $j^{\ast}(TL)$ in $\partial F\times\mathbb{R}^{2}$ divided by $q$, i.e., $\frac{1}{q}\text{winding}(j^{\ast}TL, \mathbb{R}^{2})$. We refer the reader to \cite{be} for more details.

\begin{lemma} \label{lemma: stab} \cite[Lemma 1.3]{be} Suppose the positive/negative stabilization of $L$ is $S_{\pm}(L)$.  Then we have
$$tb_{\mathbb{Q}}(S_{\pm}(L))=tb_{\mathbb{Q}}(L)-1,$$
$$rot_{\mathbb{Q}}(S_{\pm}(L), F)=rot_{\mathbb{Q}}(L, F)\pm1.$$
\end{lemma}

For $i=1,2$, suppose that $L_{i}$ is a  Legendrian knot in a contact 3-manifold $(Y_{i}, \xi_{i})$.  One can construct their connected sum, $L_{1}\sharp L_{2}$, in the contact 3-manifold $(Y_{1}\sharp Y_{2},  \xi_{1}\sharp\xi_{2})$ \cite{eh}. The following proposition generalizes \cite[Lemma 3.3]{eh}.

\begin{proposition} \label{prop: add}
For $i=1,2$, suppose that $L_{i}$ is a rationally null-homologous Legendrian knot in a contact 3-manifold $(Y_{i}, \xi_{i})$.  Then the rational Thurston-Bennequin invariant and the rational rotation number of the Legendrian knot $L_{1}\sharp L_{2}$ in the contact 3-manifold $(Y_{1}\sharp Y_{2},  \xi_{1}\sharp\xi_{2})$  satisfy
$$tb_{\mathbb{Q}}(L_{1}\sharp L_{2})=tb_{\mathbb{Q}}(L_{1})+tb_{\mathbb{Q}}(L_{2})+1,$$
$$rot_{\mathbb{Q}}(L_{1}\sharp L_{2}, F_{1}\natural F_{2})=rot_{\mathbb{Q}}(L_{1}, F_{1})+rot_{\mathbb{Q}}(L_{2}, F_{2}).$$
\end{proposition}

\begin{proof}
We denote $L_{1}\sharp L_{2}$ by $L$.  For $i=1,2$, let $p_{i}\in L_{i}$ be a point.  Suppose $(B_{i}, \xi_{i}|_{B_{i}})$ is a Darboux ball centered at $p_{i}$. That is, $B_{i}$ has coordinates $(x,y,z)$ about $p_{i}$ so that $\xi_{i}|_{B_{i}}$ is given by the one-form $dz+xdy$. Moreover, $L_{i}\cap B_{i}$ can be identified with the $y$-axis.

Since $(B_{i}, \xi_{i}|_{B_{i}})$ is a Darboux ball for $i=1,2$,  $(B_{1}, \xi_{1}|_{B_{1}})\cup (B_{2}, \xi_{2}|_{B_{2}})=(S^{3}, \xi_{std})$. Moreover, $(L_{1}\cap B_{1})\cup (L_{2}\cap B_{2})$ is a Legendrian unknot in  $(S^{3}, \xi_{std})$ with maximal Thurston-Bennequin invariant $-1$. We denote it by $U$. Its Seifert surface is a disk, we denote it by $F_{0}$.

For $i=1,2$, suppose $[L_{i}]$ is of order $q_{i}$, and $j:F_{i}\rightarrow Y_{i}$ is a rational Seifert surface of $K_{i}$,  then $j(F_{i})\cap B_{i}$ is a union of $q_{i}$ half disks with common diameter given by  $L_{i}\cap B_{i}$. For simplicity of presentation and without loss of generality, we assume that $q_1$ and $q_2$ are coprime. We choose $q_2$ copies of $j(F_{1})$ in $Y_{1}$ and  $q_1$ copies of $j(F_{2})$ in $Y_{2}$, and identify their boundaries to $L_{1}$ and $L_{2}$, respectively. We denote them by $q_{2}j(F_{1})$ and $q_{1}j(F_{2})$.  Gluing $q_{2}j(F_{1})\cap B_{1}$ and $q_{1}j(F_{2})\cap B_{2}$ along the $q_{1}q_{2}$ semi-circles  which lie in $\partial B_{1}$ and $\partial B_{2}$ respectively,  we obtain a union of $q_{1}q_{2}$ disks with common boundary $U$.   Gluing $q_{2}j(F_{1})\setminus \text{int}(B_{1})$ and $q_{1}j(F_{2})\setminus \text{int}(B_{2})$ along the $q_{1}q_{2}$ semi-circles, we obtain the image of a rational Seifert surface for $L$. We denote it by $j:F_{1}\natural F_{2}\rightarrow Y_{1}\sharp Y_{2}$.

Let $L', L'_{1}, L'_{2}$ and $U'$ be the contact push-offs of $L, L_{1}, L_{2}$ and $U$ respectively. Then we can assume that $L'\cap (Y_{1}\setminus \text{int}(B_{1}))$ coincides with $L'_{1}\setminus \text{int}(B_{1})$, $L'\cap (Y_{2}\setminus \text{int}(B_{2}))$ coincides with $L'_{2}\setminus \text{int}(B_{2})$, $U'\cap B_{1}$ coincides with $L'_{1}\cap B_{1}$, and $U'\cap B_{2}$ coincides with $L'_{2}\cap B_{2}$. So we have $$L'\cdot j(F_{1}\natural F_{2})+q_{1}q_{2} U'\cdot F_{0}=q_{2} L'_{1}\cdot j(F_{1})+q_{1} L'_{2}\cdot j(F_{2}).$$ Obviously,   $U'\cdot F_{0}=-1$. 
Hence $$tb_{\mathbb{Q}}(L)=\frac{1}{q_{1}q_{2}} L'\cdot j(F_{1}\natural F_{2}) =\frac{1}{q_{1}}L'_{1}\cdot j(F_{1})+\frac{1}{q_{2}}L'_{2} \cdot j(F_{2})+1=tb_{\mathbb{Q}}(L_{1})+tb_{\mathbb{Q}}(L_{2})+1.$$

To prove the second equality of the proposition, we choose a trivialization of  $j^{\ast}(\xi_{i})$ over $F_{i}$ for $i=1,2$; this induces a trivialization of $j^{\ast}(\xi_{1}\sharp\xi_{2})$  over $F_{1}\natural F_{2}$, and a trivialization of $j^{\ast}(\xi_{std})$  over $F_{0}$. These trivializations induce  a trivialization of $j^{\ast}(\xi_{i})$ over $\partial F_{i}$ for $i=1,2$, a trivialization of $j^{\ast}(\xi_{1}\sharp\xi_{2})$ over $\partial(F_{1}\natural F_{2})$, and a trivialization of $ \xi_{std}$ over $\partial F_{0}$. We denote them by $\partial F_{i}\times \mathbb{R}^{2}$ for $i=1,2$, $\partial (F_{1}\natural F_{2})\times \mathbb{R}^{2}$, and $\partial F_{0}\times \mathbb{R}^{2}$, respectively. 

Observe that $$\text{winding}(j^{\ast}TL, \mathbb{R}^{2})+q_{1}q_{2}\text{winding}(j^{\ast}TU, \mathbb{R}^{2})=q_{2}\text{winding}(j^{\ast}TL_{1}, \mathbb{R}^{2})+q_{1}\text{winding}(j^{\ast}TL_{2}, \mathbb{R}^{2}).$$ 
Indeed, both the left and the right sides of this equation equal $\frac{1}{2\pi}$ times the sum of the angles induced from the four Legendrian arcs $L_{1}\cap B_{1}$, $L_{2}\cap B_{2}$,  $L_{1}\setminus \text{int}(B_{1})$ and $L_{2}\setminus \text{int}(B_{2})$. For example, the Legendrian arc $L_{1}\cap B_{1}$ lift to $q_{1}q_{2}$ arcs in $\partial(F_{1}\natural F_{2})$ and  $q_{1}$ arcs in $\partial F_{1}$.  With respect to the chosen trivializations, the winding angles along the lifted arcs on both sides of the equation are the same.  

By definition, we have $$\text{winding}(j^{\ast}TU, \mathbb{R}^{2})=0,$$ $$\text{winding}(j^{\ast}TL_{1}, \mathbb{R}^{2})=q_{1}\cdot rot_{\mathbb{Q}}(L_{1}, F_{1}),$$  $$\text{winding}(j^{\ast}TL_{2}, \mathbb{R}^{2})=q_{2}\cdot rot_{\mathbb{Q}}(L_{2}, F_{2}).$$  Hence, $$rot_{\mathbb{Q}}(L, F_{1}\natural F_{2})=\frac{1}{q_{1}q_{2}}\text{winding}(j^{\ast}TL, \mathbb{R}^{2})=rot_{\mathbb{Q}}(L_{1}, F_{1})+rot_{\mathbb{Q}}(L_{2}, F_{2}).$$ \end{proof}

\section{A bound for rational Thurston-Bennequin invariants} \label{sec: bound}

Suppose $K$ is a rationally null-homologous knot in a 3-manifold $Y$; $\xi$ is a contact structure on $Y$; $L$ is a Legendrian representative of $K$ of order $q$ in $(Y,\xi)$; $F$ is a rational Seifert surface for $K$.  Using Lemma~\ref{lemma: stab}, we can perform sufficiently many times of positive stabilizations so that the contact framing of $L$ is $\lambda_{1}=\lambda_{can}+(-n+1)\mu$ without altering the number $tb_{\mathbb{Q}}(L)+ rot_{\mathbb{Q}}(L, F)$.  Performing Legendrian surgery along $L$, we obtain a contact structure $\xi_{L}$ on a 3-manifold $Y_{-n}(K)$. This Legendrian surgery induces a Stein cobordism $(W, J)$ whose concave end is  $(Y,\xi)$, and whose convex end is $(Y_{-n}(K), \xi_{L})$. Moreover, by \cite[Theorem 4.2]{r}, we have
$$\langle c_{1}(J),[\tilde{F}\cup qC]\rangle-nq-cr=2q(k_{\mathfrak{s}_{\xi}, F}+k),$$
for some integer $k$, where $\mathfrak{s}_{\xi}$ is the Spin$^{c}$ structure represented by $\xi$.

\begin{lemma}\label{lemma: rot}
$\langle c_{1}(J), [\tilde{F}\cup qC]\rangle=q\cdot rot_{\mathbb{Q}}(L, F)$.
\end{lemma}

\begin{proof}
Suppose $\xi$ is the kernel of a contact form $\alpha$ on $Y$, and $R$ is the Reeb vector field. Consider the symplectization of $(Y,\xi)$, $(Y\times [0,1], \omega=d(e^{t}\alpha))$. The restriction of the almost complex structure $J$ on $Y\times [0,1]$  is compatible with $\omega$. Moreover, $J(\xi)=\xi$, $J(R)=\partial_{t}$,  and $J(\partial_{t})=R$.  The  complex line bundle spanned by $R$ and $\partial_{t}$ can be extended to a trivial one on $W$.

By the same argument as in \cite[Proposition 2.3]{g}, the obstruction to extending a trivialization of the complex line bundle $\xi$ on $Y\times[0,1]$ to $W$ is the winding number of $j^{\ast}(TL)$ with respect to the trivialization  $\partial F\times \mathbb{R}^{2}$ induced by a trivialization of the pullback bundle $i^{\ast}\xi$ on $F$. By definition, this winding number is $q\cdot rot_{\mathbb{Q}}(L)$. Recall that $\tilde{F}$ is in fact diffeomorphic to $F$.  So $\langle c_{1}(J), [\tilde{F}\cup qC]\rangle=q\cdot rot_{\mathbb{Q}}(L)$.
\end{proof}

\begin{lemma}\label{lemma: tb}
$-nq-cr=q\cdot (tb_{\mathbb{Q}}(L)-1)$.
\end{lemma}

\begin{proof}
Recall that the contact framing of the Legendrian knot $L$ is $\lambda_{1}=\lambda_{can}+(-n+1)\mu.$  So by \cite[Page 23]{be},
\begin{align*}
tb_{\mathbb{Q}}(L)-1&=lk_{\mathbb{Q}}(K, \lambda_{1})-1=\frac{1}{q}  j(F)\cdot \lambda_{1}-1\\
&=\frac{1}{q} \cdot (q[\lambda_{can}]+cr[\mu])\cdot ([\lambda_{can}]+(-n+1)[\mu])-1\\
&=\frac{1}{q}(-nq-cr),
\end{align*}
The rational linking number, $lk_{\mathbb{Q}}(K, \lambda_{1})$, is defined in \cite[Page 21]{be}.
\end{proof}

Combining Lemma~\ref{lemma: rot} and Lemma~\ref{lemma: tb}, we get

\begin{lemma}\label{lemma:parity}
$tb_{\mathbb{Q}}(L)+ rot_{\mathbb{Q}}(L, F)=2(k_{\mathfrak{s}_{\xi}, F}+k)+1$.
\end{lemma}

\begin{proof}[Proof of Theorem~\ref{thm: main}] We proceed by a similar argument as in the proofs of \cite[Theorem 1]{p1} and \cite[Theorem 2]{h}.

The first step is to show that
\begin{equation}\label{plus1}
tb_{\mathbb{Q}}(L)+rot_{\mathbb{Q}}(L, F)\leq 2\tau^{\ast}_{c(\xi)}(Y, K, F)+1.
\end{equation}
Suppose $c(\xi_{L})\in \widehat{HF}(-Y_{-n}(K), \mathfrak{s}_{\xi_{L}})$ is the Ozsv\'{a}th-Szab\'{o} contact invariant. Let $\hat{F}_{W}: \widehat{HF}(Y, \mathfrak{s}_{\xi}) \rightarrow  \widehat{HF}(Y_{-n}(K), \mathfrak{s}_{\xi_{L}})$  and $\hat{F}_{\overline{W}}: \widehat{HF}(-Y_{-n}(K), \mathfrak{s}_{\xi_{L}})\rightarrow \widehat{HF}(-Y, \mathfrak{s}_{\xi})$  be the homomorphisms induced by the cobordisms. We have $\hat{F}_{\overline{W}}(c(\xi_{L})=c(\xi)$.  Let $\alpha$ be a homology class in $\widehat{HF}(Y, \mathfrak{s}_{\xi})$ that pairs nontrivially with $c(\xi) \in \widehat{HF}(-Y, \mathfrak{s}_{\xi})$, then   $$0\neq\langle c(\xi), \alpha\rangle=\langle\hat{F}_{\overline{W}}(c(\xi_{L})),\alpha\rangle=\langle c(\xi_{L}), \hat{F}_{W}(\alpha)\rangle.$$ So $\hat{F}_{W}(\alpha)\neq0$. By Proposition~\ref{prop: cobordismmap2}, $k_{\mathfrak{s}_{\xi}, F}+k\leq \tau^{\ast}_{c(\xi)}(Y, K, F)$.   Inequality (\ref{plus1}) then follows from Lemma~\ref{lemma:parity}.

Next we prove that
\begin{equation}\label{plus0}
tb_{\mathbb{Q}}(L)+rot_{\mathbb{Q}}(L, F)\leq 2\tau^{\ast}_{c(\xi)}(Y, K, F).
\end{equation}
We apply (\ref{plus1}) on the Legendrian connected sum of two copies of $L$, i.e., the Legendrian knot $L\sharp L\in (Y\sharp Y, \xi\sharp \xi)$: $$tb_{\mathbb{Q}}(L\sharp L)+rot_{\mathbb{Q}}(L\sharp L, F\natural F) \leq 2\tau^{\ast}_{c(\xi)\otimes c(\xi)}(Y\sharp Y, K\sharp K, F\natural F)+1.$$  Using Proposition~\ref{prop: additivity} and Proposition~\ref{prop: add}, we can rewrite the inequality as $$2tb_{\mathbb{Q}}(L)+1+2rot_{\mathbb{Q}}(L, F)\leq 4\tau^{\ast}_{c(\xi)}(Y, K, F)+1,$$  which is the same as (\ref{plus0}).

Finally, Definition \ref{def:tau*} implies that $\tau^{\ast}_{c(\xi)}(Y, K)=k_{\mathfrak{s}_{\xi}, F}+k'$ for some integer $k'$.  So (\ref{plus-1}) follows from Lemma~\ref{lemma:parity}.
\end{proof}

\section{Rational $\tau$ invariant of Floer simple knots}

Throughout this section, we will assume that the 3-manifold $Y$ is a rational homology sphere.  Thus a knot $K$ in $Y$ is automatically rationally null-homologous.  Since the Alexander grading defined by Equation (\ref{AlexGrading}) is independent of the choice of the rational Seifert surface $F$, we can conveniently suppress the subscript and write $A(\underline{\mathfrak{s}})$ for the Alexander grading.

The Alexander grading determines the genus of a knot \cite{OSzKnot} \cite{n}.  More precisely,
let $$\mathcal B_{Y,K}=\left\{\underline{\mathfrak{s}} \in \underline{\text{Spin}^{c}}(Y,K)   \left|\:\widehat{HFK}(Y,K,\underline{\mathfrak{s}})\ne0\right.\right\}.$$
If we denote $$A_{\max}=\max\{A(\underline{\mathfrak{s}})|\:\underline{\mathfrak{s}}\in \mathcal B_{Y,K}\}, \quad A_{\min}=\min\{A(\underline{\mathfrak{s}})|\: \underline{\mathfrak{s}}     \in\mathcal B_{Y,K}\},$$
then
\begin{equation}\label{AlexGenus}
A_{\max}=-A_{\min}=-\frac{\chi(F)}{2q}+\frac{1}{2},
\end{equation}
 where $F$ is a minimal genus rational Seifert surface for $K$.

Every Spin$^{c}$ structure $\mathfrak{s}$  has a conjugate Spin$^{c}$ structure $J \mathfrak{s}$ via the conjugation map $J: \text{Spin}^{c}(Y) \rightarrow \text{Spin}^{c}(Y)$.  Likewise, there is a conjugation map
$\widetilde J :      \underline{\text{Spin}^{c}}(Y,K)    \rightarrow \underline{\text{Spin}^{c}}(Y,K)$ on the set of all relative Spin$^c$ structures.   These two conjugation maps satisfy the relation
\begin{equation}\label{conjugaterelation}
G_{Y,K}(\widetilde J \underline{\mathfrak{s}})=J G_{Y,K} ( \underline{\mathfrak{s}})+ \text{PD}[K]
\end{equation}
for all $\underline{\mathfrak{s}} \in  \underline{\text{Spin}^{c}}(Y,K)$.  The
conjugation $\widetilde J$ maps $\mathcal B_{Y,K}$ into  $\mathcal B_{Y,K}$, and there is an isomorphism of absolutely graded chain complexes:
\begin{equation}\label{symmetry}
\widehat{CFK}_*(Y,K,\underline{\mathfrak{s}}   )\cong\widehat{CFK}_{*-d}(Y,K,\widetilde J\underline{\mathfrak{s}}   ),
\end{equation}
where $d=A(\underline{\mathfrak{s}})- A(\widetilde J\underline{\mathfrak{s}} )$.  Note that the Alexander grading is anti-symmetric with respect to $\widetilde J$: $$A(\underline{\mathfrak{s}})=-A(\widetilde J\underline{\mathfrak{s}}).$$ Hence, we can also write $d=2A(\underline{\mathfrak{s}})$ for the shifting of absolute grading.

\medskip
Now, assume that $K$ is a knot in an L-space $Y$.  In this special case, $\mathrm{rank} \widehat {HF}(Y, \mathfrak{s})=1$ for each Spin$^{c}$ structure $\mathfrak{s}$, so there is essentially a unique $\tau$ invariant that can be defined using the Alexander filtration described earlier.  More precisely, Let
$$\tau_\mathfrak{s}(Y,K)=\text{min}\{k_{\mathfrak{s}, F}+k|\HF(Y,\mathfrak{s}) \subset \text{Im}(I_{k})\}.$$
It is straightforward to see that $\tau_\mathfrak{s}(Y,K)$ coincides with the invariant $\tau^{\ast}_{c(\xi)}(Y, K, F)$ for nontrivial contact invariant $c_{\xi}$ by comparing its Definition \ref{def:tau*}.  \footnote{Indeed, one can also compare with other variations of $\tau$ invariant defined by Ni-Vafaee \cite{nv} and Raoux \cite{r} and find that they are all equal.}

Now, in addition, assume that $K$ is a Floer simple knot. Then there is exactly one relative Spin$^c$ structure $\underline{\mathfrak{s}}$ with underlying Spin$^c$ structure $\mathfrak{s}$ such that  $$\widehat{HFK}(Y,K,\underline{\mathfrak{s}} ) \cong \widehat {HF}(Y, \mathfrak{s})\cong \mathbb{Z}.$$  Therefore,
 \begin{equation}\label{tau=A}
\tau_{\mathfrak{s}}(Y, K)=A(\underline{\mathfrak{s}} ).
\end{equation}

Finally,  since (\ref{symmetry})  implies that $$\widehat{HFK}_m(Y,K,\underline{\mathfrak{s}}   )\cong\widehat{HFK}_{m-2A(\underline{\mathfrak{s}})}(Y,K,\widetilde J\underline{\mathfrak{s}} ) \cong \mathbb{Z} $$ for Floer simple knots, we see that the gradings of the generators must be the same as the corresponding correction terms of the underlying Spin$^c$ structures (see, e.g., \cite{NiWu}), i.e., $d(Y, G_{Y,K}(\underline{\mathfrak{s}}))=m$, $d(Y, G_{Y,K}(\widetilde J\underline{\mathfrak{s}}))=m-2A(\underline{\mathfrak{s}})$.  Hence, (\ref{conjugaterelation}) implies $$2A(\underline{\mathfrak{s}})=d(Y, \mathfrak{s})-d(Y, J\mathfrak{s}+\text{PD}[K]).$$
 See Figure \ref{fig:symm} below for a graphical illustration.

\begin{figure}[htb!]
\vspace{5pt}
\centering
\begin{tikzpicture}[scale=0.8]

    \draw[step=1, black!20!white, very thin] (-4.9, -4.9) grid (4.9, 4.9);
	
	\begin{scope}[thin, black!60!white]
		\draw [->] (-5, 0) -- (5, 0);
		\draw [->] (0, -5) -- (0, 5);
	\end{scope}

	\filldraw (0, 4) circle (2pt) node[] (a){};
	\filldraw (0, 2) circle (2pt) node[] (b){};
	\filldraw (0, -2) circle (2pt) node[] (c){};
	\filldraw (0, -4) circle (2pt) node[] (d){};
    \draw[->] (a) to [bend right=45] (d);
    \draw[->] (b) to [bend right=45] (c);
	\node [right] at (a) {$\HFK_M(Y,K,A_{\max})\cong \Z$};
	\node [right] at (b) {$\HFK_m(Y,K,A(\underline{\mathfrak{s}}))\cong \Z \cong \HF(Y, \mathfrak{s})$};
	\node [right] at (c) {$\HFK_{m-2A(\underline{\mathfrak{s}})}(Y,K,A(\widetilde J\underline{\mathfrak{s}}))\cong \Z\cong \HF(Y, J\mathfrak{s}+\mathrm{PD}[K])$};
	\node [right] at (d) {$\HFK_{M-2A_{\max}}(Y,K,A_{\min}=-A_{\max})\cong \Z$};
    \node [right] at (0,5) {$A$};
    \node at (-1.5, 0.5) {$\cong$};
    \node at (-0.5, 0.5) {$\cong$};
\end{tikzpicture}
\caption{$\HFK(Y,K)$ of a Floer simple knot in an L-space has isomorphisms $\HFK_{m}(Y,K,A(\underline{\mathfrak{s}}))\cong \HFK_{m-2A(\underline{\mathfrak{s}})}(Y,K,A(\widetilde J\underline{\mathfrak{s}}))$.  The correction terms $d(Y, \mathfrak{s})=m$, $d(Y, J\mathfrak{s}+\text{PD}[K])=m-2A(\underline{\mathfrak{s}})$.  The $\tau$ invariant  $\tau_{\mathfrak{s}}(Y, K)=A(\underline{\mathfrak{s}} ).$  }
\label{fig:symm}
\end{figure}
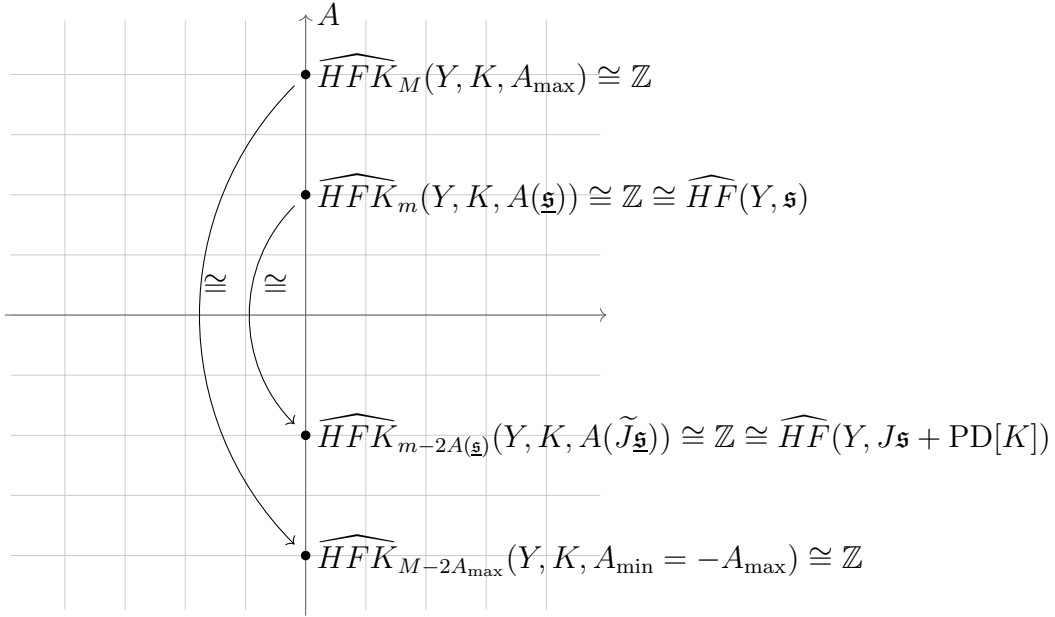

\medskip
Putting together the above discussion, we conclude that the $\tau$ invariants of a Floer simple knot $K$ in an L-space $Y$ can be determined from the correction terms of $Y$,
$$2\tau_\mathfrak{s}(Y,K) = d(Y, \mathfrak{s})-d(Y, J\mathfrak{s}+\text{PD}[K]).$$
This proves Proposition~\ref{prop:tau}.

\section{An example - simple knots in lens spaces}
As a special example, consider simple knots in lens spaces.  Remember that a lens space $L(m,n)$ is an L-space.  The notion of simple knots in lens space is describe as follows.  In Figure \ref{fig:simpleknot}, we draw the standard Heegaard diagram of a lens space $L(m,n)$.  Here the opposite side of the rectangle is identified to give a torus, and there are one $\alpha$ and one $\beta$ curve on the torus, intersecting at $m$ points and dividing the torus into $m$ regions.  We then put two base points $z$, $w$ and connect them in a proper way on the torus.  Such a simple closed curve colored in green is called a \emph{simple} knot \cite{bgh}.  There is an alternative way of describing simple knots without referring to the Heegaard diagram: Take a genus 1 Heegaard splitting $U_0 \cup U_1$ of the lens space $L(m,n)$. Let $D_0$, $D_1$ be meridian disks in $U_0$, $U_1$ such that $\partial D_0\cap \partial D_1$ consists of exactly $m$ points.  A simple knot in $L(m,n)$ is either the unknot or the union of two arcs $a_0\subset D_0$ and $a_1\subset D_1$.

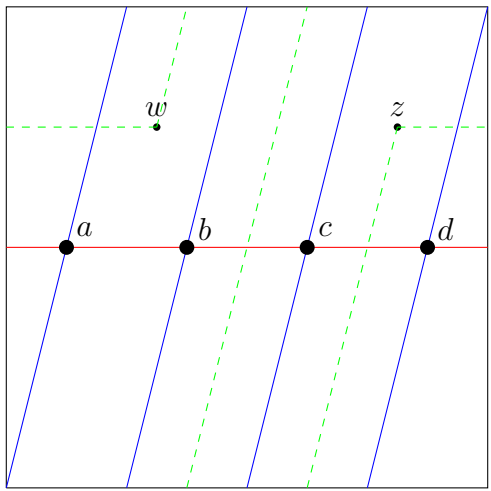
\begin{figure}[htb!]
\vspace{5pt}
\centering
\begin{tikzpicture}[scale=0.8]

\draw (-4,-4) --(-4,4) --(4,4) --(4,-4) --(-4, -4);
\draw[red] (-4,0) --(4, 0);
\draw[blue] (-4, -4) -- (-2,4);
\draw[blue] (-2, -4) -- (0, 4);
\draw[blue] (0, -4) --(2, 4);
\draw[blue] (2, -4) --( 4,4);
\draw (-3, 0) node[circle,fill,inner sep=2pt]{}; \draw (-2.7, 0.3) node{$a$};
\draw (-1, 0) node[circle,fill,inner sep=2pt]{}; \draw (-0.7, 0.3) node{$b$};
\draw (1, 0)  node[circle,fill,inner sep=2pt]{}; \draw (1.3, 0.3) node{$c$};
\draw (3, 0) node[circle,fill,inner sep=2pt]{};  \draw (3.3, 0.3) node{$d$};
\draw (-1.5, 2) node[above]{$w$} node[circle,fill,inner sep=1pt]{};
\draw (2.5, 2)  node[above]{$z$}  node[circle,fill,inner sep=1pt]{};
\draw[green, dashed] (-1.5 , 2) -- (-1, 4);
\draw[green, dashed] (-1, -4) -- (1,4);
\draw[green, dashed] (1, -4) --(2.5,2);
\draw[green, dashed] (-1.5, 2) -- (-4, 2);
\draw[green, dashed]  (2.5, 2) -- (4, 2);

\end{tikzpicture}
\caption{This is the standard Heegaard diagram of the lens space $L(4,1)$.  The red $\alpha$ curve and the blue $\beta$ curve intersect at four points $a$, $b$, $c$ and $d$. The dotted green curve is a simple knot of order 2.  }
\label{fig:simpleknot}
\end{figure}

Simple knots are Floer simple.  This follows from the observation that the knot Floer complex $\CFK(L(m,n),K)$ is generated by exactly the $m$ intersection points of $\alpha$ and $\beta$ curves.  Moreover, there is exactly one simple knot in each homology class in $H_1(L(m,n);\mathbb{Z})$ - this corresponds to the different relative positions of $z$ and $w$. Figure \ref{fig:simpleknot} exhibits a Heegaard diagram of the order 2 simple knot $K$ in the lens space $L(4,1)$.  As computed by Raoux \cite{r}, the Alexander grading of each generator is illustrated in the second row of Table \ref{table1}, which is also equal to the $\tau$ invariant of the corresponding Spin$^c$ structure. We also computed the correction terms of $L(4,1)$ using formulae in \cite[Proposition 4.8]{OSzAbGr}, and verified  $$2\tau_\mathfrak{s}(Y,K) = d(Y, \mathfrak{s})-d(Y, J\mathfrak{s}+\text{PD}[K]).$$

\begin{table}[ht!]
  \centering
  \begin{tabular}{|c | c| c| c| c |}
   \hline
   $x$ & $a$ & $b$ & $c$ & $d$  \\
   \hline
   $A(x)$& $0$ & $1/2$ & $0$ & $-1/2$ \\
   \hline
   $\tau_{\mathfrak{s}(x)}$ & $0$& $1/2$ & $0$ & $-1/2$  \\
   \hline
   $d(Y, \mathfrak{s}(x))$ & $0$ & $3/4$ & $0$ & $-1/4$\\
   \hline
   $d(Y, J\mathfrak{s}(x)+\text{PD}[K]) $ & $0$ & $-1/4$ & $0$ & $3/4$\\
   \hline
  \end{tabular}
\medskip
  \caption{For the order two simple knot $K$ in the lens space $Y=L(4,1)$, we verified that $2\tau_\mathfrak{s}(Y,K) = d(Y, \mathfrak{s})-d(Y, J\mathfrak{s}+\text{PD}[K]).$  }

\label{table1}
 \end{table}

In general, according to \cite{ho}, there are exactly $m-1$ tight contact structures on a lens space $L(m,1)$, which can be represented by Legendrian surgeries on Legendrian unknots in $(S^{3}, \xi_{std})$ with Thurston-Bennequin invariant $-m+1$, and rotation number $m-2, m-4,\cdots, 2-m$.  They bound Stein domains $(W, J_1), (W, J_2), \cdots, (W, J_{m-1})$, respectively. Since $\langle c_{1}(J_{i}), [F\cup C]\rangle=m-2i$, for $i=1,2,\cdots, m-1$, $J_1, J_2, \cdots, J_{m-1}$ represent distinct Stein structures.  By \cite[Theorem 2]{p1}, the contact invariants of these $m-1$ tight contact structures are all distinct and nontrivial. Since $L(m,1)$ is an L-space, these $m-1$ tight contact structures represent $m-1$ distinct Spin$^{c}$ structures on $L(m,1)$.

Let us turn back to the example of the order two simple knot $K$ in $L(4,1)$ depicted in Figure~\ref{fig:simpleknot}. Suppose $\xi_{1}$, $\xi_{2}$,  and $\xi_{3}$ are  the three tight contact structures on  $L(4,1)$ obtained from Legendrian surgeries on Legendrian unknots in $(S^{3}, \xi_{std})$ with Thurston-Bennequin invariant $-3$, and rotation number $2$, $0$ and $-2$, respectively. 
According to \cite{dgs}, we can compute the Hopf invariant $h(\xi_{i})$ of $\xi_{i}$, defined as $c_{1}^{2}(W,J)-2\chi(W)-3\sigma(W)$ for any Stein filling $(W,J)$ of $\xi_{i}$, and obtain that $h(\xi_{1})=h(\xi_{3})=-2$, and $h(\xi_{2})=-1$.  Recall from \cite{OSzContact} or \cite{p1} that the correction term $d(Y,\mathfrak{s}_{\xi})$ of a contact structure $\xi$ equals  $-h(\xi)/4-\frac{1}{2}$.  It follows that $d(L(4,1),\mathfrak{s}_{\xi_1})=d(L(4,1),\mathfrak{s}_{\xi_3})=0$, and $d(L(4,1),\mathfrak{s}_{\xi_2})=-\frac{1}{4}$.  Thus, we can use Table~\ref{table1} to compute the rational $\tau$-invariant of the simple knot $K$, and see that $\tau_{\xi_1}(L(4,1), K)=\tau_{\xi_3}(L(4,1), K)=0$, and $\tau_{\xi_2}(L(4,1), K)=-\frac{1}{2}$.

\bigskip

Now, suppose $\xi$ is one of the $m-1$ tight contact structures of $L(m,1)$.  Given the simple knot $K$ of order $q$ in $L(m,1)$, we compare the rational Thurston-Bennequin bound of Baker-Etnyre (\ref{BEineq}) and our bound (\ref{plus-1}) from Theorem \ref{thm: main}.

We have seen from (\ref{AlexGenus}) that the genus of a rationally null-homologous knot is determined by the Alexander grading
$$A_{\max}=-\frac{\chi(F)}{2q}+\frac{1}{2}, $$  where $F$ is a minimal genus rational Seifert surface for $K$.  So (\ref{BEineq}) implies that $$tb_{\mathbb{Q}}(L)+rot_{\mathbb{Q}}(L; F)\leq -\frac{1}{q}\chi(F)= 2A_{\max}-1.$$
Note that this bound is independent of the prescribed contact structures on the lens space.

On the other hand, it follows from (\ref{tau=A}) that $\tau_{\mathfrak{s}}=A_{\underline{\mathfrak{s}}}$ for Floer simple knots.  Thus (\ref{plus-1}) implies that
\begin{align*}
tb_{\mathbb{Q}}(L)+rot_{\mathbb{Q}}(L; F) & \leq 2\tau^{\ast}_{c(\xi)}(Y, K)-1\\
& =2A_{\underline{\xi}}-1\\
&\leq  2A_{\max}-1,
\end{align*}
where $\underline{\xi}$ is the relative Spin$^c$ structure with the underlying Spin$^c$ structure induced from the contact structure $\xi$.  (Indeed, $\tau^{\ast}_{c(\xi)}(Y, K)\leq 2A_{\max}$ is true for an arbitrary knot $K$ in a rational homology sphere $Y$. So provided that the contact invariant $c(\xi)$ is nontrivial, (\ref{plus-1}) gives a stronger bound than (\ref{BEineq}) in general.)

Finally, we remark that Cornwell obtained a Bennequin bound for lens spaces equipped with universally tight contact structures in terms of different knot invariants \cite{c}.  In contrast, our bound (\ref{plus-1}) is applicable to both universally tight and virtually overtwisted contact structures on lens spaces.  



\begin{thebibliography}{10}


\bibitem{be} K. Baker and J. Etnyre,  \emph{Rational linking and contact geometry.} Perspectives in analysis, geometry, and topology, 19-37, Progr. Math., 296, Birkh\"{a}user/Springer, New York, 2012.


\bibitem{bgh} K. Baker, J. E. Grigsby and M. Hedden, \emph{Grid diagrams for lens spaces and combinatorial knot Floer homology.} Int. Math. Res. Not. IMRN 2008, no. 10, Art. ID rnm024, 39 pp.

\bibitem{b} D. Bennequin, \emph{Entrelacements et \'{e}quations de Pfaff.}  Asterisque, 107-108:87-161, 1983.

\bibitem{c} C. Cornwell, \emph{Bennequin type inequalities in lens spacs.}  Int. Math. Res. Not. 2012, no.8, 1890-1916.

\bibitem{dgs} F. Ding, H. Geiges and A. I. Stipsicz,  \emph{Surgery diagrams for contact 3-manifolds.} Turkish J. Math. 28 (2004), no. 1, 41-74.

\bibitem{g} R. E. Gompf, \emph{Handlebody construction of Stein surfaces.}  Ann. of Math. (2) 148 (1998), no. 2, 619-693.

\bibitem{e} Y. Eliashberg, \emph{Legendrian and transversal knots in tight contact 3-manifolds.} Topological methods in modern mathematics (Stony Brook, NY, 1991), 171-193, Publish or Perish, Houston, TX, 1993.

\bibitem{eh} J. B. Etnyre and K. Honda,  \emph{On connected sums and Legendrian knots.} Adv. Math. 179 (2003), no. 1, 59-74.


\bibitem{h} M. Hedden,  \emph{An Ozsv\'{a}th-Szab\'{o} Floer homology invariant of knots in a contact manifold.} Adv. Math. 219 (2008), no. 1, 89-117.

\bibitem{ho} K. Honda, \emph{On the classification of tight contact structures I.} Geom.Topol. 4 (2000), 309-368.

\bibitem{n}  Y. Ni,  \emph{Link Floer homology detects the Thurston norm.} Geom. Topol., 13(5): 2991-3019, 2009.

\bibitem{nv} Y. Ni and F. Vafaee, \emph{Null surgery on knots in L-spaces.} arXiv:1608.07050.


\bibitem{NiWu} Y. Ni and Z. Wu, {\it Heegaard Floer correction terms and rational genus bounds}, Adv. Math. 267 (2014), 360-380.



\bibitem{OSzAbGr} P. Ozsv\'ath and Z. Szab\'o, {\it Absolutely graded Floer homologies and intersection forms for four-manifolds with boundary}, Adv. Math.  173  (2003),  no. 2, 179-261.


\bibitem{OSzFourBallGenus}  P. Ozsv\'ath and Z. Szab\'o, {\it Knot Floer homology and the four-ball genus}, Geom. Topol. 7 (2003), 615-639

\bibitem{OSzKnot}P. Ozsv\'ath and Z. Szab\'o, {\it Holomorphic disks and knot invariants},
Adv. Math. 186 (2004), no. 1, 58-116.


\bibitem{OSzContact}P. Ozsv\'ath and Z. Szab\'o, {\it Heegaard Floer homology and contact structures},
Duke Math. J. 129(2005), 39-61.





\bibitem{p} O. Plamenevskaya,  \emph{Bounds for the Thurston-Bennequin number from Floer homology.}  Algebr. Geom. Topol. 4 (2004), 399-406.

\bibitem{p1} O. Plamenevskaya,  \emph{Contact structures with distinct Heegaard Floer invariants.}  Math. Res. Lett. 11 (2004), no. 4, 547-561.

\bibitem{r} K. Raoux,  \emph{$\tau$-invariants for knots in rational homology spheres.}  arXiv:1611.09415.





\end{thebibliography}
\end{document}